\documentclass[11pt,reqno]{amsart}
\usepackage{amsmath} 
\usepackage[T1]{fontenc}
\usepackage[utf8]{inputenc}
\usepackage{lmodern}
\usepackage{microtype}
\usepackage{amsmath,amssymb,mathtools}
\usepackage{booktabs}
\usepackage{enumitem}
\usepackage{tabularx}
\usepackage{xcolor}
\usepackage{listings}
\usepackage{xurl}
\usepackage[pdfusetitle,colorlinks=true,
  pdfauthor={Junyu Guo, Hao Shen, Junqi Liu, and Lihong Zhi},
  linkcolor=blue!45!black,citecolor=blue!45!black,
  urlcolor=blue!45!black]{hyperref}

\definecolor{keywordcolor}{rgb}{0.55,0.08,0.08}
\definecolor{symbolcolor}{rgb}{0.0,0.10,0.55}
\definecolor{sortcolor}{rgb}{0.05,0.40,0.05}
\definecolor{tacticcolor}{rgb}{0.0,0.35,0.18}
\definecolor{attributecolor}{rgb}{0.55,0.20,0.0}
\definecolor{commentcolor}{rgb}{0.32,0.32,0.32}

\newtheorem{theorem}{Theorem}[section]
\newtheorem{proposition}[theorem]{Proposition}
\newtheorem{lemma}[theorem]{Lemma}

\theoremstyle{definition}

\theoremstyle{remark}

\theoremstyle{plain}
\newtheorem{problem}[theorem]{Problem}

\makeatletter
\renewcommand{\paragraph}{%
  \@startsection{paragraph}{4}{\z@}%
    {1.25ex \@plus .5ex \@minus .2ex}
    {-1em}
    {\normalfont\normalsize\bfseries}}
\makeatother
\newcommand{\M}{\ensuremath{\mathcal{M}}}
\newcommand{\NN}{\mathbb N}
\newcommand{\CC}{\mathbb C}

\newcommand{\lm}{\operatorname{lm}}
\newcommand{\lt}{\operatorname{lt}}
\newcommand{\lc}{\operatorname{lc}}
\newcommand{\inideal}{\operatorname{in}}
\newcommand{\NF}{\operatorname{NF}}
\newcommand{\wt}{\operatorname{wt}}
\newcommand{\smom}{\operatorname{sm}}

\newcommand{\modulo}[1]{\pmod {#1}}

\newcolumntype{Y}{>{\raggedright\arraybackslash}X}
\allowdisplaybreaks
\title[A  solution to Iima--Yoshino Problem 2.3]
{A Solution to Iima--Yoshino Problem 2.3\\
and Its Formalization in Lean 4}

\author{Junyu Guo}
\address{Institute of Logic and Cognition, Department of Philosophy,
Sun Yat-sen University, Guangzhou 510275, China}
\email{guojy228@mail2.sysu.edu.cn}

\author{Hao Shen}
\address{State Key Laboratory of Mathematical Sciences,
Academy of Mathematics and Systems Science, Chinese Academy of Sciences,
Beijing 100190, China; University of Chinese Academy of Sciences,
Beijing 100049, China}
\email{shenhao24@amss.ac.cn}

\author{Junqi Liu}
\address{State Key Laboratory of Mathematical Sciences,
Academy of Mathematics and Systems Science, Chinese Academy of Sciences,
Beijing 100190, China; University of Chinese Academy of Sciences,
Beijing 100049, China}
\email{liujunqi@amss.ac.cn}

\author{Lihong Zhi}
\address{State Key Laboratory of Mathematical Sciences,
Academy of Mathematics and Systems Science, Chinese Academy of Sciences,
Beijing 100190, China; University of Chinese Academy of Sciences,
Beijing 100049, China}
\email{lzhi@mmrc.iss.ac.cn}


\begin{document}

\begin{abstract}
Iima and Yoshino asked for an ideal \(I\) in
\(S=k[x_1,x_2,\ldots]\), with \(\deg x_i=i\), and a monomial order
such that
\[
 S/I\cong k[x_i:i\equiv\pm1\pmod5],
 \qquad
 \operatorname{in}(I)=(x_i^2,x_ix_{i+1}:i\geq1).
\]
We construct such an ideal and monomial order over every field \(k\)
of characteristic different from \(5\) containing an element \(c\)
with \(c^2+c=1\). The ideal has an explicit infinite homogeneous
reduced Gr\"obner basis. A five-periodic syzygy derived from a
pentagon identity proves that all basis relations belong to \(I\)
and supplies standard representations for the non-coprime critical
pairs. Triangular elimination establishes the graded quotient
isomorphism. Together, the quotient and initial ideal descriptions
yield the partition form of the first Rogers-Ramanujan identity.
In each weighted degree, a perfect matching in the support of the
normal-form matrix gives a bijection between the two partition
classes. We formalize the complex specialization in Lean~4 using
Mathlib and our set-based theory of infinite Gr\"obner bases,
including the reduced basis, the graded quotient isomorphism,
and the partition-matching theorem.
\end{abstract}

\maketitle
\enlargethispage{3pt}

\section{Introduction and main result}\label{sec:introduction}

Let \(k\) be a field and let
\begin{equation}\label{eq:ring}
  S=k[x_1,x_2,\ldots],\qquad \deg x_i=i.
\end{equation}
Iima and Yoshino developed Gr\"obner-basis theory for \(S\) and used
polynomial division to study partition identities.  They concluded
their paper with the following problem:

\begin{problem}\cite[Problem~2.3]{IimaYoshino2009} \label{prob:IY}
Find an ideal \(I\subseteq S\) and a monomial order such that
\begin{equation}\label{eq:IYproblem}
  S/I\cong k[x_i\mid i\equiv\pm1\pmod 5],
  \qquad
  \inideal(I)=J:=(x_i^2,x_ix_{i+1}\mid i\geq1).
\end{equation}
\end{problem}

The monomials of weighted degree \(n\) in the polynomial algebra on
the right-hand side of \eqref{eq:IYproblem} encode partitions of \(n\)
into parts congruent to \(1\) or \(4\) modulo \(5\), whereas the
standard monomials of the same degree modulo \(J\) encode partitions
whose consecutive parts differ by at least two. If \(I\) is
homogeneous and the quotient isomorphism preserves the grading,
the two requirements in Problem~\ref{prob:IY} imply that these
partition classes have the same cardinality for every \(n\),
giving an algebraic realization of the first Rogers-Ramanujan
identity~\cite{Andrews1998}. Our construction satisfies both
grading conditions.
\paragraph{Our contributions.}
We give an explicit solution to Problem~\ref{prob:IY} over every field
\(k\) satisfying
\begin{equation}\label{eq:field-hypothesis}
  \operatorname{char}(k)\neq5,
  \qquad
  c\in k,
  \qquad
  c^2+c=1.
\end{equation}

The hypothesis \eqref{eq:field-hypothesis} holds, for example, over
\(\mathbb Q(\sqrt5)\) and \(\mathbb C\), with
\(c=(-1\pm\sqrt5)/2\).  Define five-periodic sequences \(s_j\) and
\(\eta_j\) by
\begin{equation}\label{eq:periodic-data}
\begin{array}{c|ccccc}
j      &0&1&2&3&4\\ \hline
s_j    &2&c&-1-c&-1-c&c\\
\eta_j &0&1&c&-c&-1
\end{array}
\end{equation}
and read their subscripts modulo \(5\).  For \(n\geq2\), set
\begin{equation}\label{eq:gn-definition}
g_n=(s_n-c)x_n+
\sum_{1\leq i<n/2}s_{n-2i}x_ix_{n-i}
+\begin{cases}
x_{n/2}^2,&2\mid n,\\
0,&2\nmid n.
\end{cases}
\end{equation}
Let
\begin{equation}\label{eq:D-I-definition}
  D=\{n\geq2\mid n\equiv0,2,3\pmod5\},
  \qquad
  I=(g_n\mid n\in D).
\end{equation}
Every \(g_n\) is homogeneous of weighted degree \(n\), and hence \(I\)
is homogeneous.  Finally, put
\[
  A=\{n\geq1\mid n\equiv1,4\pmod5\},
  \qquad
  T=k[y_n\mid n\in A],
  \qquad
  \deg y_n=n.
\]

\begin{theorem}\label{thm:main}
Assume that \(k\) satisfies \eqref{eq:field-hypothesis}.  The canonical
homomorphism
\begin{equation}\label{eq:canonical-map}
  \theta:T\longrightarrow S/I,
  \qquad
  y_n\longmapsto\overline{x_n},
\end{equation}
is an isomorphism of graded \(k\)-algebras.  Moreover, for the monomial
order defined in Section~\ref{sec:groebner}, the set
\begin{equation}\label{eq:G-definition}
  \mathcal G=
  \{g_{2m}\mid m\geq1\}
  \cup
  \{c^{-1}g_{2m+1}\mid m\geq1\}
\end{equation}
is the reduced Gr\"obner basis of \(I\).  
Consequently,
\begin{equation}\label{eq:initial-ideal}
  \inideal(I)=(x_m^2,x_mx_{m+1}\mid m\geq1)=J.
\end{equation}
\end{theorem}



We have also formally verified in Lean~4 the complex specialization of
the main theorem and the partition-matching result.\footnote{\url{https://github.com/WuProver/IimaYoshinoProblem23}}
The formal proof works directly with the full infinite Gr\"obner basis and verifies the quotient
isomorphism, the reduced-basis property, the initial ideal, and the
matching theorem. Thus, the paper provides both a mathematical proof
under hypothesis~\eqref{eq:field-hypothesis} and a kernel-checked
formal proof of its specialization to \(k=\mathbb C\).


\paragraph{Related work}
Infinite Gr\"obner bases have been used to study arc spaces and
partition identities
\cite{AfsharijooMourtada2020, BaiGorskyKivinen2020,BruschekMourtadaSchepers2013}.
In characteristic zero, Bruschek, Mourtada, and
Schepers~\cite[Section~5]{BruschekMourtadaSchepers2013} obtain the
same initial ideal as in Problem~\ref{prob:IY} from the focused arc
algebra of the double point. Its quotient contains a nonzero class
\(\overline{y_1}\) satisfying \(\overline{y_1}^{\,2}=0\), so it is
nonreduced and cannot satisfy the polynomial-quotient condition
in~\eqref{eq:IYproblem}. Classical bijective proofs of the
Rogers–Ramanujan identities were given by Garsia–Milne and
Bressoud–Zeilberger \cite{BressoudZeilberger1982,GarsiaMilne1981}.
Proposition~\ref{prop:matching} obtains a partition bijection by
selecting a perfect matching in the support of the normal-form
matrix associated with our Gr\"obner basis.

Machine-checked developments of Gr\"obner bases include Buchberger's
algorithm in Coq~\cite{Thery2001}, theory refinement in
Mizar~\cite{Schwarzweller2006}, and module Gr\"obner bases and the
\(F_4\) algorithm in Isabelle/HOL~\cite{MaletzkyImmler2018}.
Our formalization builds on the Lean~4 theory
of~\cite{GuoShenLiuZhi2026}, which supports arbitrary variable-index
types and infinite sets of basis elements. This framework allows us
to work directly with the full infinite basis and verify the complex
specialization of Theorem~\ref{thm:main}, together with the
partition-matching result of Proposition~\ref{prop:matching}.

To the best of our knowledge, Theorem~\ref{thm:main} provides the
first construction of an ideal and monomial order satisfying both
requirements of Iima and Yoshino's Problem~2.3 under
hypothesis~\eqref{eq:field-hypothesis}. The field restriction is part
of the result.
\paragraph{Organization.}
Section~\ref{sec:syzygy} constructs the relations and proves their
syzygy. Section~\ref{sec:quotient} establishes the quotient
isomorphism, and Section~\ref{sec:groebner} proves the
Gr\"obner-basis assertion. Section~\ref{sec:partitions} derives
the partition correspondence, and Section~\ref{sec:lean}
describes the Lean formalization.

\section{The relations \(g_n\) and their syzygy}\label{sec:syzygy}

We first explain the origin of the polynomials \(g_n\).  Let \(K/k\)
be a field extension containing a root \(\zeta\) of \(z^2-cz+1\).
Then \(\zeta\neq0\), \(\zeta+\zeta^{-1}=c\), and \(c^2+c=1\) gives
\[
  \zeta^2+\zeta+1+\zeta^{-1}+\zeta^{-2}=0.
\]
It follows that \(\zeta^5=1\).  Since
\(\operatorname{char}(k)\neq5\), we have \(\zeta\neq1\); hence
\(\zeta\) is a primitive fifth root of unity and
\(\Delta:=\zeta-\zeta^{-1}\neq0\).  The periodic sequences in
\eqref{eq:periodic-data} can therefore be written as
\begin{equation}\label{eq:s-eta-zeta}
  s_j=\zeta^j+\zeta^{-j},
  \qquad
  \eta_j=\frac{\zeta^j-\zeta^{-j}}{\Delta}.
\end{equation}

In \(K[x_1,x_2,\ldots][[t]]\), set
\[
  F(t)=1+\sum_{n\geq1}x_nt^n
\]
and 
\begin{equation}\label{eq:R-series}
  R(t)=F(\zeta^{-1}t)F(\zeta t)-(1-c+cF(t)).
\end{equation}
Pairing the terms indexed by \(i\) and \(n-i\) and using
\eqref{eq:s-eta-zeta} gives
\begin{equation}\label{eq:R-coefficients}
  R(t)=\sum_{n\geq2}g_nt^n.
\end{equation}
Thus the family \(\{g_n\}\) is obtained from the coefficients of a
single quadratic generating-function relation.  Its main property is
the following syzygy.

\begin{proposition}[Five-periodic syzygy]\label{prop:syzygy}
For every \(N\geq2\),
\begin{equation}\label{eq:five-periodic-syzygy}
  a_Ng_N+
  \sum_{r=1}^{N-2}\eta_{4N-2r}x_rg_{N-r}=0,
  \qquad
  a_N=-(\eta_N+c\eta_{2N}).
\end{equation}
Moreover,
\begin{equation}\label{eq:aN-values}
  a_N=
  \begin{cases}
    0,&N\equiv0,2,3\pmod5,\\
    c-2,&N\equiv1\pmod5,\\
    2-c,&N\equiv4\pmod5,
  \end{cases}
\end{equation}
and, for \(1\leq r\leq N-2\),
\begin{equation}\label{eq:eta-vanishing}
  \eta_{4N-2r}=0
  \quad\Longleftrightarrow\quad
  r\equiv2N\pmod5.
\end{equation}
In particular, \(a_N\) is a unit whenever
\(N\equiv1,4\pmod5\).
\end{proposition}

\begin{proof}
Put \(\gamma=c^{-1}=1+c\), \(H=c^{-1}F\), and
\[
  E(t)=H(\zeta^{-1}t)H(\zeta t)-1-H(t)=c^{-2}R(t).
\]
For subscripts modulo \(5\), set
\[
  H_j(t)=H(\zeta^jt),
  \qquad
  E_j(t)=E(\zeta^jt).
\]
Then \(E_j=H_{j-1}H_{j+1}-1-H_j\), and direct substitution gives
\begin{equation}\label{eq:pentagon-identity}
  H_2E_4-H_3E_1-E_2+E_3=0.
\end{equation}
Writing
\[
  E(t)=\sum_{n\geq2}e_nt^n,
  \qquad e_n=c^{-2}g_n,
\]
and comparing coefficients of \(t^N\) in
\eqref{eq:pentagon-identity} gives
\[
  -\Delta(\gamma\eta_N+\eta_{2N})e_N
  +\gamma\Delta\sum_{r=1}^{N-2}
     \eta_{4N-2r}x_re_{N-r}=0.
\]
Dividing by \(\Delta\), multiplying by \(c\), and substituting
\(e_n=c^{-2}g_n\) proves \eqref{eq:five-periodic-syzygy} after
extending scalars to \(K\).  Since
\[
  S\longrightarrow K[x_1,x_2,\ldots]
\]
is injective and all coefficients in
\eqref{eq:five-periodic-syzygy} belong to \(k\), the identity descends
to \(S\).

Equations \eqref{eq:aN-values} and \eqref{eq:eta-vanishing} follow
directly from the five values in \eqref{eq:periodic-data}.  Finally,
\(2-c\neq0\), since \(c=2\) together with \(c^2+c=1\) would imply
\(5=0\).  Hence \(2-c\) and \(c-2\) are units in \(k\).
\end{proof}

The ideal \(I=\langle g_n\mid n\in D\rangle\) defined in (\ref{eq:D-I-definition})  uses only three of the
five residue classes.  The five-periodic syzygy shows that all remaining
relations already belong to \(I\).

\begin{proposition}\label{prop:all-relations}
We have
\[
  I=\langle g_n\mid n\geq2\rangle.
\]
\end{proposition}

\begin{proof}
Since \(D\subseteq\{n\geq2\}\), the inclusion
\[
  I\subseteq\langle g_n\mid n\geq2\rangle
\]
is immediate.  For the reverse inclusion, we prove \(g_N\in I\) by
strong induction on \(N\geq2\).  If \(N\in D\), this follows from the
definition of \(I\).  Otherwise \(N\equiv1\) or \(4\pmod5\), so
\(a_N\) is a unit by Proposition~\ref{prop:syzygy}.  Moreover,
\[
  2\leq N-r\leq N-1
  \qquad (1\leq r\leq N-2),
\]
and hence \(g_{N-r}\in I\) by the induction hypothesis.  Therefore
\[
  a_Ng_N
  =-\sum_{r=1}^{N-2}\eta_{4N-2r}x_rg_{N-r}\in I
\]
by \eqref{eq:five-periodic-syzygy}.  Since \(a_N\) is a unit, it
follows that \(g_N\in I\).  This completes the induction and proves
the reverse inclusion.
\end{proof}


\section{The quotient isomorphism}\label{sec:quotient}

Although Proposition~\ref{prop:all-relations} shows that
\[
 I=(g_n:n\geq2),
\]
the original generators indexed by \(D\) provide the triangular
presentation needed to identify the quotient.  Recall that
\(A\sqcup D=\NN_{>0}\).  For \(n\in D\), equation
\eqref{eq:gn-definition} has the form
\begin{equation}\label{eq:triangular-relation}
 g_n=b_nx_n+h_n(x_1,\ldots,x_{n-1}),\qquad
 b_n=
 \begin{cases}
   2-c,&n\equiv0\modulo5,\\
   -(1+2c),&n\equiv2,3\modulo5,
 \end{cases}
\end{equation}
where \(h_n\in k[x_1,\ldots,x_{n-1}]\) is homogeneous of weighted
degree \(n\).  Moreover,
\[
 (1+2c)^2=5,\qquad (2-c)(3+c)=5.
\]
Since \(\operatorname{char}(k)\neq5\), both \(1+2c\) and \(2-c\) are
units.  Hence \(b_n\in k^\times\) for every \(n\in D\).

\begin{proposition}\label{prop:quotient-isomorphism}
The homomorphism \(\theta\) in \eqref{eq:canonical-map} is an
isomorphism of graded \(k\)-algebras.
\end{proposition}

\begin{proof}
Define \(u_n\in T\) recursively by
\begin{equation}\label{eq:recursive-elimination}
 u_n=
 \begin{cases}
   y_n,&n\in A,\\
   -b_n^{-1}h_n(u_1,\ldots,u_{n-1}),&n\in D.
 \end{cases}
\end{equation}
This recursion is well defined because \(h_n\) involves only variables
with indices smaller than \(n\).  By the universal property of \(S\),
the assignments \(x_n\mapsto u_n\) define a \(k\)-algebra homomorphism
\[
 \psi:S\longrightarrow T.
\]
For every \(n\in D\),
\[
 \psi(g_n)=b_nu_n+h_n(u_1,\ldots,u_{n-1})=0.
\]
Thus \(I\subseteq\ker\psi\), and \(\psi\) induces a homomorphism
\[
 \overline\psi:S/I\longrightarrow T.
\]
If \(n\in A\), then
\[
 (\overline\psi\circ\theta)(y_n)
 =\overline\psi(\overline{x_n})=u_n=y_n.
\]
Since the \(y_n\), \(n\in A\), generate \(T\), we obtain
\(\overline\psi\circ\theta=\operatorname{id}_T\).

Conversely, we prove by induction on \(n\) that
\(\theta(u_n)=\overline{x_n}\).  This is immediate when \(n\in A\).
If \(n\in D\), then the induction hypothesis and
\eqref{eq:recursive-elimination} give
\[
\begin{aligned}
 \theta(u_n)
 &=-b_n^{-1}
   h_n\bigl(\theta(u_1),\ldots,\theta(u_{n-1})\bigr)\\
 &=-b_n^{-1}
   h_n(\overline{x_1},\ldots,\overline{x_{n-1}})
  =\overline{x_n},
\end{aligned}
\]
where the last equality follows from \(g_n=0\) in \(S/I\).
Therefore
\[
 (\theta\circ\overline\psi)(\overline{x_n})
 =\overline{x_n}
\]
for every \(n\), and hence
\(\theta\circ\overline\psi=\operatorname{id}_{S/I}\).

Finally, induction in \eqref{eq:recursive-elimination} shows that
\(u_n\) is homogeneous of weighted degree \(n\).  Consequently,
\(\overline\psi\) is graded, while \(\theta\) is graded by definition.
\end{proof}


\section{The infinite reduced Gr\"obner basis}\label{sec:groebner}

We now prove the Gr\"obner-basis assertion of
Theorem~\ref{thm:main}.  Although \(S\) has infinitely many variables,
every polynomial has finite support, and every representation below
is a finite sum.  Let
$$
 \M=\NN_{\geq0}^{(\NN_{>0})}
$$
be the additive monoid of finitely supported exponent vectors, and
write \(x^\alpha=\prod_{i\geq1}x_i^{\alpha_i}\).  A monomial order is
a total well-order on \(\M\) that is preserved by addition and has
\(0\) as its least element.

For \(\alpha\in\M\), define
\begin{equation}\label{eq:weight-moment}
 \wt(\alpha)=\sum_{i\geq1}i\alpha_i,
 \qquad
 \smom(\alpha)=\sum_{i\geq1}i^2\alpha_i.
\end{equation}
We use the same notation for monomials:
\(\wt(x^\alpha)=\wt(\alpha)\) and
\(\smom(x^\alpha)=\smom(\alpha)\).
For distinct \(\alpha,\beta\in\M\), let
$$
 r=\min\{i\geq1:\alpha_i\neq\beta_i\}.
$$
We declare \(x^\alpha\prec x^\beta\) if one of the following holds:
\begin{enumerate}[label=(\roman*),leftmargin=2.2em,itemsep=1pt]
\item \(\wt(\alpha)<\wt(\beta)\);
\item the weights agree and \(\smom(\alpha)>\smom(\beta)\);
\item the weights and second moments agree and
      \(\alpha_r<\beta_r\).
\end{enumerate}
Let \(\preceq\) be the reflexive closure of \(\prec\).  Thus, at a
fixed weight, a smaller second moment gives a larger monomial.
We use the same order notation for exponent vectors and monomials.

\begin{lemma}\label{lem:monomial-order}
The order \(\preceq\) is a monomial order.
\end{lemma}

\begin{proof}
The three successive comparisons define a total order.  It is
preserved by addition because weight and second moment are additive,
and adding the same exponent vector preserves the first differing
coordinate.

For a nonempty subset \(U\subseteq\M\), choose the least weight \(N\)
occurring in \(U\).  There are only finitely many exponent vectors of
weight \(N\), since \(\alpha_i=0\) for \(i>N\) and
\(\alpha_i\leq N/i\).  The elements of \(U\) of weight \(N\)
therefore have a least element, which is least in \(U\).
Finally, \(0\) is the unique exponent vector of weight zero.
\end{proof}

For a nonzero polynomial \(f\), let \(\lm(f)\) be the largest monomial
in its support, let \(\lc(f)\) be its coefficient, and put
\(\lt(f)=\lc(f)\lm(f)\).  For nonzero \(f,h\in S\), set
$$
 L=\operatorname{lcm}\bigl(\lm(f),\lm(h)\bigr).
$$
Their \(S\)-polynomial is
\begin{equation}\label{eq:S-polynomial}
 S(f,h)=
 \lc(f)^{-1}\frac{L}{\lm(f)}f-
 \lc(h)^{-1}\frac{L}{\lm(h)}h.
\end{equation}
For an ideal \(\mathfrak a\subseteq S\), its initial ideal is
$$
 \inideal(\mathfrak a)
 =\langle\lm(f):0\neq f\in\mathfrak a\rangle.
$$
A set \(\mathcal B\subseteq\mathfrak a\setminus\{0\}\) is a
Gr\"obner basis of \(\mathfrak a\) if
$$
 \inideal(\mathfrak a)=\langle\lm(g):g\in\mathcal B\rangle.
$$
A standard representation of \(p\) below a monomial \(d\), with
respect to \(\mathcal B\), is an equality
$$
 p=\sum_{g\in F}a_gg,\qquad
 F\subseteq\mathcal B\text{ finite},\qquad a_g\in S,
$$
such that \(\lm(a_gg)\prec d\) for every nonzero summand.
Buchberger's criterion also applies to infinite
sets~\cite[Proposition~1.13 and its proof]{IimaYoshino2009}: a monic generating set is a
Gr\"obner basis if every \(S\)-polynomial of distinct generators has
a standard representation below the corresponding least common
multiple.  Pairs with relatively prime leading monomials satisfy
this condition by the product criterion.

\begin{lemma}\label{lem:leading-monomials}
For every \(m\geq1\),
$$
 \lm(g_{2m})=x_m^2,\qquad
 \lm(g_{2m+1})=x_mx_{m+1},
$$
with leading coefficients \(1\) and \(c\), respectively.
\end{lemma}

\begin{proof}
All terms of \(g_n\) have weight \(n\).  Among its quadratic
monomials \(x_ix_{n-i}\), with \(1\leq i\leq n/2\), the second moment
\(i^2+(n-i)^2\) is uniquely minimized when the indices are as close
as possible.  The resulting monomials are \(x_m^2\) and
\(x_mx_{m+1}\), with nonzero coefficients \(1\) and \(s_1=c\).
The linear monomial \(x_n\), when present, has larger second moment
and is therefore smaller.
\end{proof}

For \(n\geq2\), put
$$
 q_n=\lm(g_n)
     =x_{\lfloor n/2\rfloor}x_{\lceil n/2\rceil}.
$$
Inspection of their variable supports shows that the distinct
non-coprime pairs are, up to interchanging the entries,
\begin{equation}\label{eq:overlap-families}
 (q_{2m},q_{2m+1}),\qquad
 (q_{2m+1},q_{2m+2}),\qquad
 (q_{2m-1},q_{2m+1}),
\end{equation}
where \(m\geq1\) in the first two families and \(m\geq2\) in the
third.

\begin{proposition}\label{prop:groebner}
The set \(\mathcal G\) in \eqref{eq:G-definition} is the reduced
Gr\"obner basis of \(I\).  Consequently,
$$
 \inideal(I)=(x_m^2,x_mx_{m+1}:m\geq1)=J.
$$
\end{proposition}

\begin{proof}
Since \(c(c+1)=1\), the element \(c\) is a unit.
Proposition~\ref{prop:all-relations} and
Lemma~\ref{lem:leading-monomials} therefore show that
\(\mathcal G\) is a monic generating set of \(I\).
By the product criterion, it remains to consider the three families
in \eqref{eq:overlap-families}.

For \(1\leq r\leq N-2\), the leading monomial of \(x_rg_{N-r}\)
has a second moment
$$
 \rho_N(r)=r^2+
 \left\lfloor\frac{N-r}{2}\right\rfloor^2+
 \left\lceil\frac{N-r}{2}\right\rceil^2
 =r^2+\left\lceil\frac{(N-r)^2}{2}\right\rceil.
$$
For \(1\leq r\leq N-3\), direct subtraction gives
\begin{equation}\label{eq:rho-difference}
 \rho_N(r+1)-\rho_N(r)=
 \begin{cases}
  3r+2-N,&\text{if }N-r\text{ is even},\\
  3r+1-N,&\text{if }N-r\text{ is odd}.
 \end{cases}
\end{equation}
When \(N=3m+1\) or \(N=3m+2\), with \(m\geq1\), these differences
show that the minimum occurs exactly at \(r=m,m+1\).
When \(N=3m\), with \(m\geq2\), the unique minimum occurs at \(r=m\);
after this value is removed, the minimum occurs exactly at
\(r=m-1,m+1\), since the function decreases before \(m\), increases
after \(m\), and
$$
 \rho_N(m-1)=\rho_N(m+1)=\rho_N(m)+2.
$$


When \(N=3m\), the central term \(r=m\) vanishes because
\(\eta_{4N-2m}=\eta_{10m}=0\). Consequently, the two nonzero product
terms sharing the largest leading monomial are those indexed by
\(r=m-1\) and \(r=m+1\).

Using the five values of \(\eta\), the two nonzero product terms
sharing the largest leading monomial in each relevant syzygy
have the following sums:
$$
\begin{array}{c|c|l}
 N&r&\text{sum of the two terms}\\ \hline
 3m+1&m,m+1&
 -x_mg_{2m+1}+cx_{m+1}g_{2m}
 =cS(g_{2m},g_{2m+1})\\[2pt]
 3m+2&m,m+1&
 -cx_mg_{2m+2}+x_{m+1}g_{2m+1}
 =cS(g_{2m+1},g_{2m+2})\\[2pt]
 3m&m-1,m+1&
 cx_{m-1}g_{2m+1}-cx_{m+1}g_{2m-1}
 =-c^2S(g_{2m-1},g_{2m+1}).
\end{array}
$$
Their common leading monomials are the corresponding least common
multiples
$$
 x_m^2x_{m+1},\qquad
 x_mx_{m+1}^2,\qquad
 x_{m-1}x_mx_{m+1},
$$
whose second moments are, respectively,
$$
 3m^2+2m+1,\qquad
 3m^2+4m+2,\qquad
 3m^2+2.
$$
Every other nonzero product term in the syzygy has larger second
moment by \eqref{eq:rho-difference}.  For the stated ranges of \(m\),
each displayed value is also strictly smaller than
\(N^2/2\leq\lceil N^2/2\rceil=\smom(q_N)\).
Thus, the standalone term \(a_Ng_N\), when nonzero, has a larger second
moment as well.  Since all these monomials have weight \(N\), their
larger second moments place them strictly below the corresponding
least common multiple.

For each row, let \(\mathcal R\) be its pair of indices, and let
\(\lambda=c,c,-c^2\), respectively.  Solving
\eqref{eq:five-periodic-syzygy} for the indicated \(S\)-polynomial
gives
$$
 S(f,h)=-\lambda^{-1}\left(
 a_Ng_N+
 \sum_{\substack{1\leq r\leq N-2\\r\notin\mathcal R}}
 \eta_{4N-2r}x_rg_{N-r}\right),
$$
where \(f,h\) are the pair shown in that row.  Every nonzero summand
on the right has a leading monomial strictly below
\(\operatorname{lcm}(\lm(f),\lm(h))\).
Nonzero rescaling of either argument leaves its \(S\)-polynomial
unchanged.  Expressing each odd \(g_j\) as \(c(c^{-1}g_j)\) therefore
gives standard representations with respect to \(\mathcal G\).
Buchberger's criterion proves that \(\mathcal G\) is a Gr\"obner
basis, and its leading monomials generate \(J\).

It remains to check reducedness.  Every nonleading quadratic
monomial has two distinct indices differing by at least two.
It is therefore divisible by neither a square nor a product of
consecutive variables.  Linear tail monomials are also standard
modulo \(J\).  Moreover, the \(q_n\) are pairwise distinct quadratics,
so none divides another.  Since \(\mathcal G\) is monic, it is
reduced.
\end{proof}

\begin{proof}[Completion of the proof of Theorem~\ref{thm:main}]
Proposition~\ref{prop:quotient-isomorphism} proves the graded
quotient isomorphism, and Proposition~\ref{prop:groebner} proves
the reduced Gr\"obner-basis and initial-ideal assertions.
\end{proof}

\section{The Rogers-Ramanujan correspondence}\label{sec:partitions}

For \(n\geq0\), a partition of \(n\) is a multiplicity vector
\(a=(a_i)_{i\geq1}\in\M\) satisfying
\[
  \sum_{i\geq1}ia_i=n.
\]
Let \(P(n)\) consist of the partitions with \(a_i=0\) unless
\(i\equiv1,4\pmod5\), and let \(Q(n)\) consist of those that satisfy
\begin{equation}\label{eq:Q-condition}
 a_i\leq1,\qquad a_i\neq0\Longrightarrow a_{i+1}=0
 \quad(i\geq1).
\end{equation}
Thus the parts of a partition in \(Q(n)\) differ by at least two.
Both sets are finite.  The monomials of weighted degree \(n\) in
\(T\) are indexed by \(P(n)\), while a monomial \(x^a\) lies outside
\(J\) precisely when \eqref{eq:Q-condition} holds.

Let \(S_n\) and \(T_n\) denote the components of weighted degree \(n\),
and put
$$
 V_n=\operatorname{span}_k\{x^\mu:\mu\in Q(n)\}\subseteq S_n.
$$
Since \(\mathcal G\) is a homogeneous Gr\"obner basis with initial
ideal \(J\), every class in \((S/I)_n\) has a unique representative
in \(V_n\) \cite[Proposition~1.14]{IimaYoshino2009}.  For \(f\in S_n\), denote this representative by
\(\NF_{\mathcal G}(f)\).  The residue map
\(V_n\to(S/I)_n\) is therefore a linear isomorphism.  Composing its
inverse with the degree-\(n\) part of \(\theta:T\to S/I\) gives a
linear isomorphism
\begin{equation}\label{eq:degree-normal-form-map}
 \Phi_n:T_n\longrightarrow V_n,\qquad
 y^\lambda\longmapsto\NF_{\mathcal G}(x^\lambda).
\end{equation}
Here, for \(\lambda\in P(n)\),
$$
 y^\lambda=\prod_{i\in A}y_i^{\lambda_i},
 \qquad
 x^\lambda=\prod_{i\in A}x_i^{\lambda_i}.
$$
The monomial bases of \(T_n\) and \(V_n\) are indexed by \(P(n)\)
and \(Q(n)\), respectively.

\begin{proposition}\label{prop:matching}
For \(\lambda\in P(n)\) and \(\mu\in Q(n)\), define
\begin{equation}\label{eq:normal-form-matrix}
 M_n(\mu,\lambda)
 =[x^\mu]\NF_{\mathcal G}(x^\lambda).
\end{equation}
After choosing orderings of \(P(n)\) and \(Q(n)\), the matrix \(M_n\)
is square and invertible.  There exists a bijection
\(\pi_n:P(n)\to Q(n)\) such that
\begin{equation}\label{eq:matching-property}
 [x^{\pi_n(\lambda)}]\NF_{\mathcal G}(x^\lambda)\neq0
 \qquad(\lambda\in P(n)).
\end{equation}
\end{proposition}

\begin{proof}
The matrix \(M_n\) represents the linear isomorphism \(\Phi_n\).
Hence \(d:=|P(n)|=|Q(n)|\) and \(\det M_n\neq0\).
Write the chosen orderings as
\(P(n)=\{\lambda_1,\ldots,\lambda_d\}\) and
\(Q(n)=\{\mu_1,\ldots,\mu_d\}\).
The determinant expansion
$$
 \det M_n=
 \sum_{\sigma\in\mathfrak S_d}
 \operatorname{sgn}(\sigma)
 \prod_{j=1}^d M_n(\mu_{\sigma(j)},\lambda_j)
$$
has at least one nonzero product.  Choose such a permutation
\(\sigma\) and set
\(\pi_n(\lambda_j)=\mu_{\sigma(j)}\).
This is a bijection, and each selected matrix entry is nonzero,
which proves \eqref{eq:matching-property}.
\end{proof}

Taking dimensions in each degree gives the Hilbert-series identity \cite[Corollary~1.15]{IimaYoshino2009}
\begin{equation}\label{eq:hilbert-series-preview}
\begin{aligned}
 \sum_{n\geq0}|P(n)|t^n
 &=\prod_{\substack{j\geq1\\j\equiv1,4\pmod5}}
   \frac{1}{1-t^j}
 =\operatorname{Hilb}_{S/I}(t)\\
 &=\operatorname{Hilb}_{S/J}(t)
 =\sum_{n\geq0}|Q(n)|t^n.
\end{aligned}
\end{equation}
This is an identity in \(\mathbb Z[[t]]\). Thus, the equality of
partition numbers is an equality of integers, including when \(k\)
has positive characteristic.

To obtain the usual series-product form, consider a partition in
\(Q(n)\) with \(\ell\) parts.  Subtracting the staircase
\((2\ell-1,2\ell-3,\ldots,1)\) leaves an ordinary partition with
at most \(\ell\) parts and decreases the weight by \(\ell^2\).
Consequently, \eqref{eq:hilbert-series-preview} becomes the first
Rogers–Ramanujan identity
$$
 \sum_{\ell\geq0}
 \frac{t^{\ell^2}}{\prod_{j=1}^{\ell}(1-t^j)}
 =
 \prod_{\substack{j\geq1\\j\equiv1,4\pmod5}}
 \frac{1}{1-t^j},
$$
where the empty product is \(1\).
Thus Theorem~\ref{thm:main} yields both the identity and a partition
bijection without assuming the Rogers–Ramanujan identity.

The bijection can be obtained by finite procedures.  In weighted
degree \(n\), only relations \(g_j\) with \(2\leq j\leq n\) can be
used in reduction.  Over a coefficient field with effective
arithmetic and zero-testing, with the chosen \(c\) given effectively,
one can therefore compute \(M_n\) and select a permutation with a
nonzero product of entries.  Fixing the orderings and choosing the
lexicographically first such permutation specifies \(\pi_n\).

An individual normal form may contain several standard monomials;
the permutation selects one from each column without repeating a
row.  For example, put \(\delta=1+2c\).  Using \(g_2,g_4,g_5\) and
\(c^{-1}=1+c\), reduction gives
\begin{equation}\label{eq:degree-five-normal-form}
 \NF_{\mathcal G}(x_1^5)
 =\frac{\delta^3}{c^3}x_1x_4-\frac{25}{c}x_5.
\end{equation}
Both coefficients are nonzero because \(\delta^2=5\) and
\(\operatorname{char}(k)\neq5\).  In parts notation,
$$
 P(5)=\{(1,1,1,1,1),(4,1)\},
 \qquad
 Q(5)=\{(5),(4,1)\}.
$$
Since \(\NF_{\mathcal G}(x_1x_4)=x_1x_4\), these orderings give
$$
 M_5=
 \begin{pmatrix}
  -25/c&0\\
  \delta^3/c^3&1
 \end{pmatrix}.
$$
The second column forces \((4,1)\) to match with \((4,1)\).
The unique matching supported by this matrix is therefore
$$
 (1,1,1,1,1)\longmapsto(5),
 \qquad
 (4,1)\longmapsto(4,1).
$$
In general, the construction consists of polynomial division followed
by a finite matching among the nonzero normal-form coefficients.

\section{Formalization in Lean 4}\label{sec:lean}

We have formalized the complex specialization of
Theorem~\ref{thm:main} and the partition bijection of
Proposition~\ref{prop:matching} in Lean~4~\cite{deMouraUllrich2021},
using Mathlib~\cite{Mathlib2020}.\footnote{\url{https://github.com/WuProver/IimaYoshinoProblem23}}
The development builds on our earlier formalization of
Gr\"obner-basis theory~\cite{GuoShenLiuZhi2026}.  Its support for
arbitrary variable-index types and infinite sets of basis elements
allows us to work directly with \(\CC[x_1,x_2,\ldots]\) and the full
family \(\mathcal G\).  In particular, we reuse the infinite
Buchberger criterion, polynomial division, and normal-form theory.

The polynomial ring is represented by
\texttt{MvPolynomial PNat Complex}, abbreviated as
\texttt{CaseI.S}.  In the namespace \texttt{CaseI}, the main
declaration corresponding to Theorem~\ref{thm:main} is:
\begin{lstlisting}
theorem theorem_1_2 (c : ℂ) (hc : c ^ 2 + c = 1) :
    ∃ hG : caseIMonomialOrder.IsGroebnerBasis (G c) (I c),
      hG.IsReduced ∧
      Ideal.span
        (caseIMonomialOrder.leadingTerm '' (I c : Set S)) = J ∧
      (∀ i : ℕ+,
        caseIMonomialOrder.degree (g c (2 * (i : ℕ))) =
          Finsupp.single i 2) ∧
      (∀ i : ℕ+,
        caseIMonomialOrder.degree (g c (2 * (i : ℕ) + 1)) =
          Finsupp.single i 1 + Finsupp.single (next i) 1) ∧
      Function.Bijective (allowedToQuotient c)
\end{lstlisting}
The witness \texttt{hG} certifies the Gr\"obner-basis property.
The remaining assertions give reducedness, the initial ideal,
the two leading-monomial formulas, and bijectivity of the canonical
algebra homomorphism.  Here
\texttt{caseIMonomialOrder.degree} denotes the leading exponent
vector.  Grading preservation is proved separately.

The proof verifies the five-periodic syzygy and uses its three
overlap families to construct standard representations of the
\(S\)-polynomials.  The existing formal Buchberger criterion then
establishes the Gr\"obner-basis property.  The quotient isomorphism
is verified through recursive elimination and mutually inverse
homomorphisms.  Thus the general infinite-variable theory supplies
the algebraic infrastructure, while the present development proves
the identities and constructions specific to this example.

The declaration corresponding to Proposition~\ref{prop:matching} is:
\begin{lstlisting}
theorem proposition_5_1 (c : ℂ) (hc : c ^ 2 + c = 1)
    (hG : caseIMonomialOrder.IsGroebnerBasis (G c) (I c)) (hred : hG.IsReduced) (n : ℕ) :
    ∃ π : P n ≃ Q n,
      ∀ lam : P n,
        (normalForm hG hred (partitionMonomial lam.val)).coeff
          (Partition.multiplicities (π lam).val) ≠ 0
\end{lstlisting}
Here \texttt{P n} and \texttt{Q n} are finite types of partitions,
and \texttt{P n} $\simeq$ \texttt{Q n} represents a bijection.  The coefficient
condition is precisely \eqref{eq:matching-property}.  The proof
constructs the two monomial bases and selects a nonzero permutation
 term in the determinant of the normal-form matrix.  
The names

The development also verifies the explicit choice
\(c=(-1+\sqrt5)/2\), yielding complex results without a remaining
hypothesis on \(c\).  The partition result is existential; an
executable matching algorithm is not implemented.  The general-field
theorem and the infinite generating-series identity are not included
in the Lean formalization.  The checked source uses Lean
\texttt{v4.34.0-rc1} and Mathlib revision \texttt{062f1e3da807}.
An axiom audit confirms that the main results depend only on Lean's
standard logical axioms, with no dependence on \texttt{sorryAx}.

\section{Conclusion}\label{sec:conclusion}

We have constructed an explicit ideal and monomial order satisfying
both requirements of Iima--Yoshino's problem under the field hypothesis
\eqref{eq:field-hypothesis}.  The five-periodic syzygy supplies the
omitted relations and the standard representations required by
Buchberger's criterion, while triangular elimination identifies the
quotient as a polynomial algebra.  The resulting normal-form matrices
yield bijections between \(P(n)\) and \(Q(n)\), giving an algebraic
proof of the first Rogers-Ramanujan identity.  The complex
specialization and the partition bijections have been formally
verified in Lean~4, using our earlier formalization of
infinite-variable Gr\"obner-basis theory.

Three directions remain for further work.  Can one construct a solution over fields
such as \(\mathbb Q\), where \(z^2+z-1\) has no root, or in
characteristic \(5\)?  Can the bijection in
Proposition~\ref{prop:matching} be described by a direct combinatorial
rule that satisfies the same normal-form support condition?
Finally, extending the Lean formalization to all fields satisfying
\eqref{eq:field-hypothesis} would bring its scope into agreement with
the mathematical theorem.

\section*{Acknowledgements}
GPT-5.6 sol proposed the construction presented in this paper during
discussions with the authors and assisted in drafting and revising
the manuscript. The authors have carefully verified the mathematical
statements and proofs and reviewed and revised the final text.
The complex specialization and the partition-matching theorem have
also been formally verified in Lean~4. The authors take full
responsibility for the content of the paper.
\bibliographystyle{amsplain}
\bibliography{references}

\end{document}